\documentclass[10pt]{amsart}
\usepackage{amsmath}
\usepackage{amsthm}
\usepackage{amssymb}

\def\R{{\hbox{\bf R}}}
\def\C{{\hbox{\bf C}}}

\def\P{{\hbox{\bf P}}}

\font \roman = cmr10 at 10 true pt
\def\ep{{\epsilon}}

\def\emph#1{{\it #1}}
\def\textbf#1{{\bf #1}}
\def\rank{{\hbox{\roman rank}}}
\def\CQ{{\mathcal Q}}
\theoremstyle{plain}
\newtheorem{theorem}[subsection]{Theorem}
\newtheorem{conjecture}[subsection]{Conjecture}

\newtheorem{lemma}[subsection]{Lemma}
\newtheorem{corollary}[subsection]{Corollary}

 \theoremstyle{remark}
\newtheorem{remark}[subsection]{Remark}

\theoremstyle{definition}
\newtheorem{definition}[subsection]{Definition}

\begin{document}

\title{On the rank of random sparse matrices}

\author{Kevin P. Costello}
\address{Department of Mathematics, Institute for Advanced Study, Princeton, NJ 08540}
\email{kcostell@math.ias.edu}
\thanks{Kevin Costello is supported by NSF Grant DMS-0635607}
\author{Van  Vu}
\address{Department of Mathematics, Rutgers, Piscataway, NJ 08854}
\email{vanvu@math.rutgers.edu}
\thanks{Van Vu is partially supported by  NSF Career Grant 0635606 and the Simonyi Foundation}

\begin{abstract}
We investigate the rank of random (symmetric) sparse matrices. Our
main finding  is that with high probability, any dependency that
occurs in such  a matrix is formed by  a set of few rows that
contains an overwhelming number of zeros. This allows us to obtain
an exact estimate for the co-rank.

\end{abstract}
\maketitle

\section{Introduction and Statement of Main Results}

Let
 $w_{ij}, 1\le i \le j \le n$ be complex
numbers, where the $w_{ij}$ are non-zero for $i<j$
 (we place no restriction on the $w_{ii}$). We denote by
 $W$ the symmetric matrix where

 \begin{itemize}

 \item The upper diagonal entries are
 $w_{ij}, 1 \le i < j\le n$.

 \item The diagonal entries are $w_{ii}$.

 \item $w_{ji} = w_{ij}$ for all $1 \le i < j \le n$.

 \end{itemize}

 Let $p$ (which may depend on $n$)
 be a positive number at most one and  $\xi_{ij}$, $1\le i \le j \le n$ be
  independent Bernoulli random variables taking on value 1 with
  probability $p$ and 0 with probability $1-p$.

We sparsify the  matrix $W$  by
  replacing each entry above or on the diagonal by zero with probability
  $1-p$ while keeping the matrix symmetric. The resulting
  symmetric matrix, denoted by $Q:= Q(W,p)$ has entries $q_{ij}=
  w_{ij}\xi_{ij}, 1\le i \le j \le n$ and $q_{ji}= q_{ij}$.

$Q$ is a very general model of random matrices and the study of its
linear algebraic parameters is of considerable interest. In this
paper, we will focus on the rank of $Q$. This seemingly simple
parameter has not been understood until very recently, when in
\cite{CTV} T. Tao and the authors proved (in a slightly less general
form) that if $0<p<1$ is a constant not depending on $n$, then $Q$
almost surely has full rank, thus proving a conjecture of Weiss.
 A refinement of this result in \cite{CV} showed that the same
 conclusion still holds as far as $p > \frac{(1+ \epsilon) \ln n}{n}$. This
 bound is sharp, as for $p < \frac{(1-\epsilon) \ln n}{n}$, by the coupon-collector theorem
 there will be an all-zero row with probability tending to 1 as $n$ tends to
 infinity.

In the current paper, we investigate the case when $p$ can be
significantly smaller than $\frac{\ln n}{n}$. In this case, the
co-rank of the matrix is almost surely positive (in fact, polynomial
in $n$) and this makes the situation quite different from the
full-rank case considered before.

The main intuition that underlines  many problems concerning the
rank of a random matrix is that

\vskip2mm

\centerline {\it Dependency should come from small configurations.}

\vskip2mm

 The most famous question based on this intuition is the
$(1/2)^{n}$ conjecture, which asserts that the probability that a
random Bernoulli matrix  of order $n$ (non-symmetric matrix with
entries being iid Bernoulli random variables, taking values $1$ and
$-1$ with probability $1/2$) is singular is $(1/2+o(1))^{n}.$ The
lower bound is trivial, as it comes from the probability that there
are two equal rows. Despite several efforts (see \cite{Vusur} for a
recent survey), this conjecture is still open. Here and later, the
asymptotic notation will be used under the assumption that $n$, the
size of the matrix, tends to infinity.

 In this paper, we are able, for the first time, verify this
 intuition rigorously in a sufficiently general setting.
 We will show that for a certain range of $p$, any dependency that occurs in the random
 matrix indeed comes from a small configuration. To be more precise,
we will prove that if $p=\Omega(\frac{\ln n}{n})$, then any dependency
comes from a collection of at most $k$ rows, for some $k=O(1)$,
which between them have at most $k-1$ non-zero columns. In particular, it will
follow that the co-rank of $Q$ depends only on the local structure
of an underlying graph formed by the $\xi_{ij}$, regardless of the
values of the $w_{ij}$. (It is, however, critical that $w_{ij} \neq 0$ for $i \neq j$.)

In order to state our results formally, we will need few
definitions:

\begin{definition}
The graph $G(Q)$ of a symmetric matrix $Q$ is the graph with
vertices being the rows of $Q$ and  vertex $i$ is connected to
 vertex $j$ if and only if $q_{ij} \neq 0$.
\end{definition}

Note that this graph may contain self-loops if the entries on the
main diagonal of $Q$ are nonzero.  Furthermore, except for these
loops the graph of $Q(W,p)$ is independent of $W$ due to our
assumption that $w_{ij} \neq 0$.

 For a set $S$ in a graph $G$, let $N(S)$ denote the neighborhood of $S$, that is, the set of vertices
adjacent to $S$. (We allow $N(S)$ to contain elements of $S$.) A set
$S$ of vertices of a graph $G$ is \textbf{non-expanding} if it
satisfies $|N(S)|<|S|$.  $S$ is \textbf{minimally non-expanding} if
it contains no proper non-expanding subset.

{\it Examples.} An isolated vertex forms a non-expanding set. A set
of two vertices of degree one sharing a common neighbor forms a
minimal  non-expanding set of size 2.  However, a single vertex with
a self-loop does not form a non-expanding set (as it is its own
neighborhood).

A set $T$ of rows of a matrix $Q$ is non-expanding if at most
$|T|-1$ columns of $Q$ have at least one nonzero entry in $T$. Note
that a set of rows of $Q$ is non-expanding if and only if the
corresponding set of vertices in $G(Q)$ is non-expanding.
Furthermore, if $S$ is a non-expanding subset of $G(Q)$, then the
corresponding rows in $Q$ are dependent (regardless of $W$), since
they span a space of dimension at most $|N(S)|$. It is thus clear
that for any set $S$ of vertices of $G(Q)$ and any weights $W$ that

$$\rank (Q) \le n -|S| + |N(S)|. $$

By considering all sets of vertices in $G(Q)$ we immediately obtain the upper bound

$$\rank (Q) \le \min_{S \subset V(G(Q))} n - |S|  + |N(S)|. $$

Our first main result shows that for $p$ which is sufficiently large, this upper bound is {\it tight}.

\begin{theorem} \label{thm:altexpression}
Assume that $p = \Omega (\frac{\ln n}{n})$. Let $W$ be any fixed
symmetric matrix whose off-diagonal entries are nonzero. Then almost
surely

\begin{equation} \label{eqn:rankbound}
\rank (Q(W, p))=\min_{S \subset V(G(Q))} (n-|S|+|N(S)|).
\end{equation}
\end{theorem}

In other words, the rank of $Q$ is determined, almost surely, by the
local structure of $G(Q)$, which in turn depends only on the $\xi$
and which diagonal entries of $W$ are nonzero (actually, which
diagonal entries are nonzero will also turn out to be with high
probability irrelevant to the rank of $Q$--see Remark
\ref{remark:smallsets}).


 Theorem \ref{thm:altexpression}, together with earlier results from \cite{CTV,
 CV}, imply the following corollary:

\begin{corollary} \label{cor:randomw}
 Let $1/2>p=\Omega(\frac{\ln n}{n})$, and let $Q$ be a random symmetric
 matrix whose above diagonal entries are independent
  random variables $\xi_{ij}$ satisfying
\begin{equation*}
\P(\xi_{ij}=0)=1-p
\end{equation*}
and whose off-diagonal entries are 0 with probability at least $1-p$.

Then almost surely
\begin{equation} \label{eqn:genrankbound}
\rank (Q) =\min_{S \subset V(G(Q))} n-|S|+|N(S)|)
\end{equation}
\end{corollary}


The $\xi_{ij}$ here are not necessarily identical or Bernoulli.

\vskip2mm

We next turn to the question of classifying the dependency in
$Q(W,p)$. It is a routine to check  that for any fixed positive
integer $s$ and $\ep>0$, if $p<\frac{(1-\ep) \ln n}{sn}$ then
$G(n,p)$ almost surely contains minimal non-expanding sets of all
sizes up to and including $s$. In $Q(W,p)$ such a set of rows is
automatically dependent, regardless of the values of the nonzero
entries.

Our second main result shows that this is the {\it only } reason for dependency.

\begin{theorem} \label{thm:betterexpression}
Let $s$ be a fixed positive integer and $c >1/s$ a positive
constant. Assume that $\frac{c \ln n}{n}<p<1/2$. Then with
probability $1-O(\frac{1}{(\ln \ln n)^{1/4}})=1-o(1)$,
 $Q(W,p)$ has the property that any set of dependent rows of
$Q(W,p)$  contains a non-expanding set of size at most $s-1$.  The
constant implicit in the $O$ notation is independent of $W$.
\end{theorem}
Note that in particular the $s=1$ case of this result gives that $Q(W,p)$ is almost surely nonsingular for $p$ sufficiently large.  This particular case can also be extended to non-symmetric models (see Section 11).   Again, this result can easily be extended to a result for random $W$ in the vein of \ref{cor:randomw}.

The special cases $s=1,2$ of this result were proved in an earlier
paper \cite{CV} for $Q=Q(n,p)$. This theorem implies that  if
$\frac{c \ln n}{n}<p<1/2$, where $c>1/s$ for some positive integer
$s$, then any minimal non-expanding set has size at most $s-1$. It
also implies the following corollary.

\begin{definition}
A set $S$ of vertices of a graph $G$ is \textbf{s-unobstructed} if
it contains no non-expanding subset of size at most $s$. $S$ is
\textbf{unobstructed} if it contains no non-expanding subset.
\end{definition}

\begin{corollary}
With probability $1-O(\frac{1}{(\ln \ln n)^{1/4}})=1-o(1)$, the rank of the random matrix $Q(W,p)$
 equals the size of its largest unobstructed set in its graph.
\end{corollary}

The proofs of the main theorems combine techniques developed in
earlier papers \cite{CTV,CV} with some delicate properties of sparse
random graphs. We are going to sketch the main ideas in the next
section.

\section{The Idea of the proofs and some Lemmas}
Instead of proving Theorems  \ref{thm:altexpression} and \ref{thm:betterexpression}
directly, we are going to prove the following theorem (which is somewhat weaker
than Theorem \ref{thm:betterexpression}). The proof of this theorem,
combined with some lemmas will imply Theorems
 \ref{thm:altexpression} and \ref{thm:betterexpression}.

 We say that a matrix $Q$ is $s$-{\bf saturated} if the rank of $Q$
 equals the size of the largest $s$-unobstructed set of $G(Q)$.

\begin{theorem} \label{thm:mainresult} Let $s$ be a fixed positive integer and
$c >1s$ a positive constant. Assume $\frac{c \ln n}{n}<p<1/2$. Then
with probability $1-O(\frac{1}{(\ln \ln n)^{1/4}})=1-o(1)$, the
random matrix $Q(W,p)$ is $(s-1)$-saturated.
\end{theorem}

Since the case where $p=\Omega(1)$ was already addressed in \cite{CTV}, we will assume that $p=o(1)$.

Following the ideas from \cite{Kom1, CTV}, we are going to expose
$Q(W,p)$ minor by minor. Letting $Q_m$ denote the upper left $m
\times m$ minor of
$%
Q(W,p)$, we view $Q_{m+1}$ as being formed by taking $Q_m$ and
augmenting by a column whose entries are chosen independently, along
with the column's transpose.  Let $G_m=G(Q_m)$. In graph theoretic terms, we can view the sequence of $G_m$ as a vertex exposure process of $G(n,p)$.

Our starting observation is that when a good portion of the vertices
have been exposed, the rank of the matrix is  close to its size.

Recall that $p \ge c\ln n/ n$ for a constant $c >1/s$. Let
$0 <\delta<1$ be a constant such that $1/s<\delta c < 1/(s-1)$. Define
$n^{\prime}:=\delta n$.

\begin{lemma}\label{lemma:nearfullrank}
For any constant $\epsilon>0$ there exists a constant $\gamma>0$
such that
\begin{equation*}
{\hbox{\bf P}}(\hbox{rank}(Q_{n^{\prime}})<(1-\epsilon)n^{\prime})=o(e^{-%
\gamma n \ln n})
\end{equation*}
\end{lemma}

Our plan is to show that the addition of the remaining
$n-n^{\prime}$
rows/columns is enough to remove all the linear dependencies  from $%
Q_{n^{\prime}}$, {\it except} those corresponding to non-expanding
subsets of  at most $s-1$ vertices.

The next several lemmas provide some properties of the (random) graph $G_m$ for $n^{\prime}\leq
m \leq n$.

\begin{definition}
A graph $G$ is {\bf well-separated} if the following two conditions
hold:

W1. Any connected subgraph of $G$ on at most $5s$ vertices contains
at most $s-1$ vertices with degree at most $\ln \ln n$.

W2. No cycle of length 1 or between 3 and $12s$ in $G$ contains a vertex of
degree at most $\ln \ln n$.
\end{definition}

\begin{lemma}

\label{lemma:separation} For any constant $\ep >0$, the probability that there is an $m$ between $n'$ and $n$ for which $G_m$ is not well separated is
$O(n^{-sc \delta + 1+ \ep})$.
\end{lemma}

Note that by our choice of $\delta$ this probability will be $o(1)$
for sufficiently small $\ep$.

\begin{remark} \label{remark:smallsets}
One corollary of this lemma and our main results is that whether the
diagonal entries of $W$ are non-zero is likely to be irrelevant to
the collection of dependent sets in $Q(W,p)$.  The minimal dependent sets
will with high probability be the non-expanding sets of size at most
$s-1$ and therefore correspond to vertices of degree at most $s-1$
in $G$.  On the other hand, a graph satisfying $W2$ only contains
self-loops at vertices of degree at least $\ln \ln n$.

Another way of thinking about this is as follows: If we form a new
matrix $W'$ by replacing all of the entries on the diagonal of $W$
by 0, with high probability the graph of $Q(W',p)$ will contain
exactly the same non-expanding sets as the graph of $Q(W,p)$ and
thus (by Theorem \ref{thm:betterexpression}) the same dependent sets
of rows.
\end{remark}

\begin{definition}
A graph $G$ is a {\bf small set expander} if every subset $S$ of the
vertices of $G$ with $|S| \leq \frac{n}{\ln^{3/2} n}$ either has at
least $|S|$ edges connecting $S$ to $\bar{S}$, its complement, or has
a subset $S' \subset S$ with $|S'| \leq s-1$ and at most $|S'|-1$
edges connecting $S'$ to $\bar{S'}$.
\end{definition}

\begin{lemma} \label{lemma:expansion}
For any $m>n'$ the probability that $G_m$ is well separated but is
not a small set expander is $O(n^{-4})$.
\end{lemma}

\begin{definition}
\label{def:distinguished} A set $S$ of the vertices of a graph $G$
is {\bf
nice} if there are at least two vertices of $G$ each is adjacent to
exactly one vertex in $S$.

A set $S$ of the vertices of a graph $G$ is {\bf nearly nice} if there is at least one vertex in $G$ which has exactly one neighbor in $S$.
\end{definition}

Set $k := {\frac{\ln \ln n }{2 p}}$.  We will next define a class of 'good' matrices which behave well under augmentation.

\begin{definition}

\label{def:goodgraph} A graph $G$ is {\bf good} if the following four
properties hold:

1. Every minimal non-nice subset of the vertices of $G$
either has size at least $k+1$ or contains a non-expanding subset of
size at most $s-1$.

2. Every minimal non-nearly nice subset of the vertices of $G$
either has size at least $k+1$ or is
a non-expanding set of size at most $s-1$.

3. At most $\frac{1}{p \ln n}$ vertices of $G$ have degree less than
$s$.

4. $G$ is well separated.

A symmetric matrix $Q$ is {\bf good} if the graph $G(Q)$ is good.
\end{definition}

The next lemma states that in the  augmentation process we will
likely run only into good matrices.
\begin{lemma} \label{lemma:goodmatrices} Let $\ep$ be a positive constant.
Then with probability $1-O(n^{1-sc\delta+\ep})$, $Q_m$ is good for
every $m$ between $n'$ and $n$ .
\end{lemma}

We now consider the effect of augmentation on the rank of $A$ when $A$ is a good matrix.

\begin{definition} \label{def:normalpair}
A pair $(G, G')$ of graphs is called {\bf normal} if the following properties hold:

1. $G$ is an induced subgraph on $|G'|-1$ vertices of $G'$.

2. The new vertex added to $G$ is not adjacent to any vertex which was part of a non-nearly nice subset in $G'$.

A pair $(A, A')$ of symmetric matrices is normal if the pair of graphs $(Q(A), Q(A'))$ is normal.
\end{definition}

\begin{lemma} \label{lemma:singaug} Let $Q$ be any fixed, good $m \times
m$ matrix which is not $s-1$-saturated.  Then
\begin{equation*}
{\hbox{\bf P}}(\,\, \hbox{rank}(Q_{m+1})-\,\, \hbox{rank}(Q_m)<2 |
(Q_m, Q_{m+1}) \, \, \hbox{is normal} \wedge Q_m=Q)=O((kp)^{-1/2}).
\end{equation*}
\end{lemma}

What the above lemma says is that if $Q_m$ is good but not
$s-1$-saturated, then augmenting it will tend to remove some of the
dependencies among the rows of $Q_m$ (note that $kp$ is tending to
infinity by assumption).  If $(Q_{m+1}, Q_m)$ is normal, then the
dependencies removed can't correspond to small non-expanding subsets
of the rows of $Q_m$.  This implies that in some sense $Q_{m+1}$ is
a little closer to being saturated then $Q_m$ was.

Now suppose on the other hand that $Q_m$ is both good and already
$s-1$-saturated.  We are going to show that  (again assuming no
change in the non-expanding subsets) with high probability $Q_m$
does not gain any new dependencies by being augmented.

\begin{lemma} \label{lemma:nonsingaug} Let $Q$ be any fixed, good $m \times m$ matrix which is also $s-1$-saturated.  Then
\begin{equation*}
{\hbox{\bf P}}(\,\, \hbox{rank}(Q_{m+1})- \,\, \hbox{rank}(Q_m)<2 |
(Q_m, Q_{m+1}) \, \, \hbox{is normal} \wedge Q_m=Q)=O((kp)^{-1/4}).
\end{equation*}
\end{lemma}

ln the next section, we  prove Theorem \ref{thm:mainresult}
assuming these lemmas.  The proofs of the lemmas will be presented in the sections that follow.

\section{Proof of Theorem \ref{thm:mainresult}}
In this section, we  assume all lemmas from the previous section are true. We are going to use
a variant of an argument from \cite{CTV}. Let $B_0$ be the event
that $G_n$ is $(s-1)-$saturated.  Let $B_1$ be the event that the
rank of $Q_{n^{\prime}}$ is at least $n^{\prime}(1-{\frac{1-\delta
}{4 \delta}})$.  Let $B_2$ be the event that $Q_m$ is good for all
$n^{\prime}\leq m < n$.  By Bayes' theorem we have
\begin{equation*}
{\hbox{\bf P}}(B_0) \leq {\hbox{\bf P}}(B_0 \wedge B_2 | B_1) +
{\hbox{\bf P}}(\neg B_1) + {\hbox{\bf P}}(\neg B_2)
\end{equation*}
By Lemma \ref{lemma:nearfullrank} we have that ${\hbox{\bf P}}(\neg
B_1)=o(e^{-\gamma n \ln n})$ and by Lemma \ref{lemma:goodmatrices}
we have that ${\hbox{\hbox{\bf P}}(\neg B_2)=O(n^{1-sc\delta
+\ep})}$. Both of these probabilities are much smaller than the
bound $O((\ln \ln n)^{-1/4})$ which  we are trying to prove, so it only
remains to bound the first term.

Let $U_m$ denote the size of the largest $(s-1)-$unobstructed subset
of the vertices of $G_m$.

Let $Y_m=U_m-\,\, \hbox{rank}(Q_m)$.  Our goal is now to prove that $Y_n$ is almost surely 0.
Define a random variable $X_m$ as follows:
\begin{itemize}
\item  $X_m=4^{Y_m}$ if $Y_m>0$ and every $Q_j$ with $n' \leq j
\leq m$ is good;
\item  $X_m=0$ otherwise.
\end{itemize}
The core of the proof is the following  bound on  the expectation of
$X_{m+1}$ given any fixed sequence  $\CQ_m$ of matrices $\{Q_{n'},
Q_{n'+1}, \dots, Q_m\}$ encountered in the augmentation process.
\begin{lemma} \label{lemma:expectationofX} For any sequence $ \CQ_m=
\{Q_{n'}, Q_{n'+1}, \dots, Q_m\}$ encountered in the augmentation
process,
\begin{equation*}
{\hbox{\bf E}} (X_{m+1} |  \CQ_m  ) < \frac{3}{5} X_m + O((\ln
\ln n)^{-1/4}).
\end{equation*}
\end{lemma}
Let us (for now) assume Lemma \ref{lemma:expectationofX} to be true.
This lemma together with Bayes theorem shows that for $n^{\prime}<m$
we have
\begin{equation*}
{\hbox{\bf E}}(X_{m+1} | Q_{n'}) < {\frac{3 }{5}}{\hbox{\bf E}}(X_m |
Q_{n'})+O((\ln \ln n)^{-1/4}).
\end{equation*}
By induction on $m_2-m_1$ we now have that for any $m_2\geq m_1 \geq
n^{\prime}$
\begin{equation*}
{\hbox{\bf E}}(X_{m_2} | Q_{n'})<(\frac{3}{5})^{m_2-m_1} {\hbox{\bf E}}%
(X_{m_1} | Q_{n'})+O((\ln \ln n)^{-1/4}).
\end{equation*}
In particular, by taking $m_2=n$ and $m_1=n^{\prime}$ we get that
\begin{equation*}
{\hbox{\bf E}}(X_n | Q_{n'})<(\frac{3}{5})^{n-n'} X_{n^{\prime}} +
O((\ln \ln n)^{-1/4}).
\end{equation*}
\noindent If $Q_{n'}$ satisfies $B_1$, we automatically have $X_{n'}
\leq 4^{\frac{(1-\delta)n'}{4 \delta}}=(\sqrt{2})^{n-n'}$, so
\begin{equation*}
{\hbox{\bf E}}(X_n | Q_{n'}) <(\frac{3 \sqrt{2}}{5})^{n-n'} + O((\ln
\ln n)^{-1/4})=O((\ln \ln n)^{-1/4}).
\end{equation*}
\noindent By Markov's inequality, for any $Q_{n'}$ satisfying $B_1$
\begin{equation*}
{\hbox{\bf P}}(X_n>3 | Q_{n'})=O((\ln \ln n)^{-1/4})
\end{equation*}
\noindent On the other hand, by definition $X_n\ge 4$ if $G_n$ is
not $(s-1)-$saturated and $B_2$ holds. It thus follows by summing
over all $Q_{n'}$ satisfying $B_1$ that
\begin{equation*}
{\hbox{\bf P}}(B_0 \wedge B_2 | B_1) =O((\ln \ln n)^{-1/4}),
\end{equation*}
\noindent proving the theorem. \vskip2mm It remains to prove Lemma
\ref{lemma:expectationofX}. If a matrix in the sequence $\CQ_m=
\{Q_{n'}, Q_{n'+1}, \dots, Q_m\}$ is not good, then $X_{m+1}=0$ by
definition and there is nothing to prove. Thus, from now on we can
assume that all matrices in the sequence are good. Let $Z_m$ denote
the number of vertices of degree at most $s$ in $G_{m}$ adjacent to
the $m+1^{st}$ vertex of $G_{m+1}$.

\textbf{Claim:} $U_{m+1}-U_m \leq Z_m+1$.

\begin{proof}(of claim): Let $S_{m+1}$ denote a $s-$unobstructed subset of the vertices of $G_{m+1}$ such that $|S_{m+1}|=U_{m+1}$.  Let $S_{m}'$ denote the set formed by removing the $m+1^{st}$ vertex from $S$, as well as any vertices of degree at most $s$ adjacent to that new vertex.

$S_m'$ is $s-$unobstructed since each subset of $S_m'$ of size at most $s$ either contains a vertex of degree at least $s+1$ (in which case it clearly expands) or has the same neighborhood in $G_m$ as in $G_{m+1}$.  Since at most $Z_m+1$ vertices were removed to go from $S_{m+1}$ to $S_m'$, the claim follows.
\end{proof}

By the above claim, if $Z_m$ is positive, then
augmenting the matrix will increase $Y_m$ by at most $Z_m+1$ ($U_m$ increases by at most $Z_m+1$ and the rank does not
decrease). Furthermore, $Z_m=0$ if and only if $(Q_m, Q_{m+1})$ is
normal. By Bayes' theorem, we have
\begin{eqnarray*}
{\hbox{\bf E}}(X_{m+1}| \CQ_m) &=& {\hbox{\bf E}} (X_{m+1}\chi( Z_m >0)| \CQ_m)
   \\&\,\,+&{\hbox{\bf E}}
(X_{m+1}| \CQ_m \wedge (Q_m, Q_{m+1}) \, \, \hbox{is normal})
{\hbox{\bf
P}} ((Q_m, Q_{m+1}) \, \, \hbox{is normal}| \CQ_m)  \\
&\leq& {\hbox{\bf E}} (X_{m+1} \, \chi (Z_m>0)| \CQ_m) +{\hbox{\bf
E}} (X_{m+1}| \CQ_m \wedge (Q_m, Q_{m+1}) \, \, \hbox{is normal})
\\
&=&  {\hbox{\bf E}} (4^{Z_m +1+Y_m} \chi (Z_m>0) |\CQ_m) +
{\hbox{\bf E}} (X_{m+1}| \CQ_m \wedge (Q_m, Q_{m+1}) \, \, \hbox{is
normal}) .
\end{eqnarray*}
Since $Q_m$ is good, $G_m$ has at most $\frac{1}{p \ln n}$ vertices
which have degree at most $s$. Thus, we can bound
$Z_m$ by the sum of $\frac{1}{p \ln n}$ random Bernoulli variables,
each of which is 1 with probability $p$. It follows that
\begin{equation*}
\hbox{\bf P} (Z_m = i) \leq \binom{(p \ln n)^{-1}}{i} p^{i} \leq
(\ln n)^{-i}.
\end{equation*}
Adding up over all $i$, we have
\begin{equation*}
\hbox{\bf E} (4^{Z_m + 1} \chi(Z_m>0 )|\CQ_m) \leq
\sum_{i=1}^{\infty} 4^{i+1} (\ln n)^{-i} = O(({\ln n})^{-1}).
\end{equation*}
If $Y_m=0$ and $(Q_m, Q_{m+1})$ is normal, then  by Lemma
\ref{lemma:nonsingaug} (which applies since $Q_m$ is good) $X_{m+1}$
is either 0 or 4, with the probability of the latter being $O((\ln
\ln n)^{-1/4})$. Therefore we have for any sequence $\CQ_m=
\{Q_{n'}, \dots Q_m\}$ of good matrices with $Y_m=0$ that
\begin{equation} \label{fullrankexpec}
{\hbox{\bf E}}(X_{m+1} | \CQ_m)=O((\ln \ln n)^{-1/4}+(\ln
n)^{-1})=O((\ln \ln n)^{-1/4}).
\end{equation}
If $Y_m=j>0$ and $(Q_m, Q_{m+1})$ is normal, then $Y_{m+1}$ is $j-1$
with probability $\newline 1-O((\ln \ln n)^{-1/2})$ by Lemma
\ref{lemma:singaug}, and otherwise is at most $j+1 $.  Combining
this with the bound on $\hbox{\bf E} (4^{Z_m + 1} \chi(Z_m>0
)|\CQ_m)$ we have
\begin{equation} \label{nonfullrankexpec}
{\hbox{\bf E}} (X_{m+1}|\CQ_m)=4^{j-1}+4^{j+1}
O((\ln \ln n)^{-1/2})+4^j O((\ln n)^{-1}) \leq {\frac{3 }{5}%
}4^j
\end{equation}
The lemma now follows immediately from  (\ref{fullrankexpec}) and
(\ref{nonfullrankexpec}).

\section{Proof of Lemma \protect\ref{lemma:nearfullrank}}
We use a variant of an argument from \cite{CV}.  By symmetry and the union bound
  $$\P( \text{rank%
}(Q_{n^{\prime}})<(1-\epsilon) n^{\prime}) \le
\binom{n^{\prime}}{\epsilon n^{\prime}} \times \P(B_1^*),$$ where
$B_1^*$ denotes the event  that the last $\epsilon n^{\prime}$
columns of $Q_n^{\prime}$ are contained in the span of the remaining
columns.

We view $Q_{n'}$ as a block matrix,
$$Q_{n'}=\left[ \begin{array}{c|c} A & B \\ \hline B^T & C \\ \end{array}
\right],$$ where $A$ is the upper left $(1-\ep)n' \times (1-\ep)n'$
sub-matrix and $C$ has dimension $\ep n' \times \ep n'$.  We obtain
an upper bound on $\P(B_1^*)$ by bounding the probability of $B_1^*$
conditioned on any fixed $A$ and $B$ (treating $C$ as random).

$B_1^*$ cannot hold unless the columns of $B$ are contained in the
span of those of $A$, meaning the equation $B=A F$ holds for some (not necessarily unique)
matrix $F$.  If this is the case, then $B_1^*$ will hold only when
we also have $C=B^T F$. This means that each entry of $C$ is forced
by our choice of $A$, $B$ and our assumption that $B_1^*$ holds.

However, $C$ is still random, and the probability that any given
off-diagonal entry takes on its forced value is at most $1-p$ (this being the probability that the entry is 0).  The entries are
not all independent (due to the symmetry of $C$), but those
above the main diagonal are. Therefore the probability that $B_1^*$
holds for any fixed $A$ and $B$ is at most $(1-p)^{\frac{(\ep
n')^2}{2}}$.

We therefore have

\begin{eqnarray*}
{\hbox{\bf P}} (\text{rank}(Q_{n^{\prime}})<(1-\epsilon) n^{\prime})
&\leq& \binom{n^{\prime}}{\epsilon n^{\prime}}
((1-p)^{\frac{(\epsilon n^{\prime})^2 }{2}}) \\
&\leq& ({\frac{n^{\prime}e }{\epsilon n^{\prime}}})^{\epsilon
n^{\prime}} e^{\frac{-
p (\epsilon n^{\prime})^2 }{2}} \\
&\leq& {c_2}^{n} e^{-c_1 n \ln n}.
\end{eqnarray*}

where $c_1$ and $c_2$ are positive constants depending on
$\epsilon$, $\delta $, and $c$ (but independent of $n$).

\section{Proof of Lemma \ref{lemma:separation}}
If $p$ is at least $(\ln n)^2 / n$ then $G_m$ will with
probability at least $1-o(1/n^3)$ have no vertices with degree at
most $\ln \ln n$, in which case the lemma is trivially  true.
Therefore we can assume $p \leq (\ln n)^2 / n$

If $G_{m}$ fails to be well separated for some $m$ between $n'$ and
$n$ there must be a first $m_0$ with this property. We are going to
bound the probability that a fixed $m$ is this $m_0$.

Case 1: $m_0=n'$.  We can bound the probability $G_{n'}$ fails
condition W1 by the union bound over all sets of at most $5s$
vertices of the probability that those vertices form a connected
subgraph with at least $s$ small-degree vertices.

The probability that any single vertex has sufficiently small degree
is at most
\begin{equation} \label{eqn:lowdegree}
\sum_{i=0}^{\ln \ln n} \binom{n'-1}{i} p^i (1-p)^{{n'}-i} \leq
(1+o(1))\sum_{i=0}^{\ln \ln n} (n'p)^i (1-p)^{n'} \leq \frac{(\ln
n)^{2\ln \ln n}}{n^{c \delta}},
\end{equation}

so the probability that a set of size $i$ contains at least $s$ such
vertices is at most
\begin{equation*}
\binom{i}{s} n^{-s c \delta + \ep}.
\end{equation*}

The probability that a set of size $i$ is connected is (by the union
bound over all spanning trees) at most
\begin{equation*}
i^{i-2} p^{i-1} \leq \frac{(\ln n)^{2i}}{(n')^{i-1}}
\end{equation*}

By the FKG inequality \cite{FKG} these two events are negatively correlated (as one is monotone increasing under edge inclusion, while the other is monotone decreasing), so
the probability that some subset fails the first well-separation
criterion is at most
\begin{equation} \label{wellsepfailure1}
\sum_{i=s}^{5s} \binom{n'}{i} \binom{i}{s} n^{-s c \delta +\ep}
\frac{(\ln n)^{2i}}{(n')^{i-1}} = O(n^{1-sc \delta +\ep}).
\end{equation}

Similarly, for each $1 \leq j \leq 12s$ the probability that a given
set of size $j$ contains a spanning cycle is by the union bound at
most
\begin{equation*}
\frac{(i-1)!}{2} p^i \leq \frac{(i-1)! (\ln n)^{2i}}{n^i}
\end{equation*}

and by the FKG inequality this event is negatively correlated with
the set containing a vertex of degree at most $\ln \ln n$ in $G$.
Therefore the probability some set fails W2 is at most
\begin{equation} \label{wellsepfailure2}
\sum_{i=1}^{12s} \binom{n'}{i}\frac{(i-1)! (\ln n)^{2i}}{n^i}
\frac{(\ln n)^{2 \ln \ln n}}{n^{c \delta}}=O(n^{-c \delta + \ep})
\end{equation}

Case 2: $m_0=m>n'$.  In this case we can bound the probability that
$m_0=m$ by the probability  that $G_{m-1}$ is well separated but
$G_m$ fails to be well separated.  As in the previous case we can
take a union bound over all sets of at most $5s$ vertices, but now
we need only consider sets which contain the vertex newly added to
create $G_m$ (all other sets are covered by our assumption that
$G_{m-1}$ is well separated).  This means that the probability of
failure of either requirement for any particular $m$ in this range
is at most $12s/n$ times the corresponding union bound in
\eqref{wellsepfailure1} and \eqref{wellsepfailure2}, which is $O(n^{-s
c\delta + \ep})$.

By the union bound, the property that $G_{m}$ is not well -separated for some $m$ is at most

$$O(n^{1-sc \delta +\ep}) + O(n^{-c \delta + \ep}) + n \times O(n^{-s
c\delta + \ep})=  O(n^{1-sc \delta +\ep}) , $$

\noindent completing the proof.

\section{Proof of Lemma \ref{lemma:expansion}}

In order to prove the edge expansion property we first show that
almost surely all small subgraphs of $G(n,\frac{c \ln n}{n})$ will
not have too many edges.

\begin{definition} \label{def:locally sparse}
A graph $G$ is {\bf locally sparse} if every subgraph on at
most $\frac{n}{\ln^{3/2} n}$ vertices has average degree less than
8.
\end{definition}

\begin{lemma} \label{lemma:sparseness}

For fixed $c$ the probability that $G(n,\frac{c \ln n}{n})$ is not
locally sparse is $O(n^{-4})$.
\end{lemma}

\begin{proof}
Let $q_j$ be the probability that a subset of size $j$ induces at least
$4j$ edges. By the union bound, this is at most $\binom{n}{j}$ times
the probability of a particular subset inducing at least $4j$ edges,
so

\begin{eqnarray*}
q_j &\leq& \binom{n}{j} \binom{j^2 /2}{4j} p^{4j} \\
&\leq& (\frac {ne}{j})^j  (\frac{e j c \ln n}{8 n})^{4j} \\
&\leq& (\frac{c^4 e^5 j^3 \ln ^4 n }{n^3})^{j} .\\
\end{eqnarray*}

For $j<n^{1/4}$ this gives $q_j \leq n^{-2j}$, while for $j>n^{1/4}$
we have (using our upper bound on $j$) $q_j \leq (\ln n)^{-j/2}
=o(n^{-5})$.  By adding up over all $j$ at least 2, we conclude that
the failure probability is $O(n^{-4})$, completing the proof.
\end{proof}

Armed with this lemma we can now prove Lemma \ref{lemma:expansion}, which we do in two cases depending on the value of $p$.

Case 1: $p \geq \frac{12 \ln n}{n}$: We estimate the probability
that there is a non-expanding small set directly by using the
union bound over all sets of size $i<n \ln^{-3/2} n$. The
probability in question can be bounded from above by

\begin{eqnarray*}
\sum_{i=1}^{n \ln^{-3/2} n}
\binom{n}{i}\binom{i(n-i)}{i-1}(1-p)^{i(n-i)-(i-1)} &\leq&
\sum_{i=1}^{n \ln^{-3/2} n} n^i
(en)^{i-1} e^{-i n p (1+o(1))} \\
&=& \frac{1}{en} \sum_{i=1}^{n \ln^{-3/2} n} (n^2
e^{-np(1+o(1))})^i.
\end{eqnarray*}

The lower bound on $p$ guarantees that the summand is
$O(n^{-(4+o(1))i})$, so the probability for any $p$ in this range is
$o(\frac{1}{n^4})$.

Case 2: $p<\frac{12 \ln n}{n}$.

If $G_m$ fails to expand edgewise there must be a minimal subset
$S_0$ which both fails to expand and contains no non-expanding
subset of size at most $s-1$.

We claim that the subgraph formed by the vertices of $S_0$ must have
average degree at least 8.  For the sake of contradiction, let us suppose that this were not the case.
Because fewer than $|S_0|$ edges leave $S_0$, it follows that the
average degree of the vertices of $S_0$ in $G$ is at most 9, meaning
that at most $\frac{9|S_0|}{\ln \ln n}$ vertices in $S_0$ have
degree at least $\ln \ln n$ in $G$.

We next consider the connected components of the induced subgraph of
$G$ on the vertices in $S_0$.  Unless $G$ fails to satisfy condition
W1 (in which case we are done), at most $\frac{9(s-1) |S_0|}{\ln \ln
n}$ of the vertices with degree at most $\ln \ln n$ can be in the
same component as a vertex with degree at least $\ln \ln n$.

Furthermore, the remaining vertices must be in components of size at
most $s-1$ (again due to W1).  Let $T$ be one of those components.
By assumption, $T$ has at least $|T|$ edges leaving $T$, and each of
these edges must also leave $S_0$.  But this implies that $S_0 - T$
is a smaller edgewise non-expanding set, a contradiction.

What we have actually shown in the above is the following \textit{deterministic} statement: any locally sparse, well-separated graph must also be a small set expander.  We can therefore bound the probability of the existence of a minimal
$S_0$ by the probability of $G$ failing to be locally sparse, which
by Lemma \ref{lemma:sparseness} is $O(n^{-4})$.

\section{Proof of Lemma
\protect\ref{lemma:goodmatrices}} Let $C_0$ be the event that $G_m$
is good for every $m$ between $n'$ and $n$. Let $C_1$ be the event
that $G_m$ has at most $\frac{1}{p \ln n}$ vertices of degree less
than $s$ for every $m$ between $n'$ and $n$, $C_2$ be the event that
$G_m$ has maximum degree at most $5 k n p$ for each $m$, and $C_3$
be the event that $G_m$ is well separated, locally sparse, and a
small set expander for every $m$ between $n'$ and $n$. We have
$${\hbox{\bf P}}(\neg C_0) \leq {\hbox{\bf P}}(\neg C_0 \wedge C_1 \wedge
C_2 \wedge C_3)+ {\hbox{\bf P}}(\neg C_1) + {\hbox{\bf P}}(\neg C_2)
+ {\hbox{\bf P}}(\neg C_3).$$  We are going to bound each term on the
right hand side separately, in reverse order.

Lemmas \ref{lemma:expansion}, \ref{lemma:separation}, and
\ref{lemma:sparseness}  together show that ${\hbox{\bf P}}(\neg
C_3)=O(n^{1-s c\delta +\ep})$.

${\hbox{\bf P}}(\neg C_2)$ is at most the expected number of
vertices of degree at least $5 k n p$ in $G_n$, which is at most
$$n \binom{n}{5 k n p} p^{5 k n p} \leq n (e/5 k n p )^{5 k np}
p^{5 k n p} \leq n e^{-5 k n p} =o(n^{-4}).$$

To bound $\P( \neg C_1)$, we note that the probability that some
$G_m$ contains a set of vertices of degree less than $s$ of  size $t=
p^{-1} \ln^{-1} n$ is bounded from above by the probability that at
least $t$ vertices in $G_n$ each have at most $s$ neighbors amongst
the vertices of $G_{n'}$.  This probability is clearly decreasing in $p$, and for $p=\frac{\ln n}{sn}$ it is by Markov's inequality at most
\begin{eqnarray*}
\frac{n}{t} \sum_{i=1}^{s} \binom{n'}{i} p^i (1-p)^{n'-i}&<& n \frac{(\ln n)^2}{sn} \sum_{i=1}^s n^i \frac{(\ln n)^i}{(sn)^i} e^{-(1+o(1)) \ln n' / s} \\
&=& O((\ln n)^2 n^{-1/s} \sum_{i=1}^s (\ln n)^i) = n^{-1/s+o(1)}
\end{eqnarray*}

It remains to estimate the first term, which we will do by the union
bound over all $m$.  Since property $C_1$ implies $G_m$ has few
vertices of small degree, it suffices to estimate the probability
that $G_m$ contains a non-nice set while still satisfying properties
$C_1$, $C_2$, and $C_3$. Let $p_{j}$ be the probability that
conditions $C_1, C_2$, and $C_3$ hold but some subset of $j$
vertices causes $G_m$ to fail to be good. Symmetry
and the union bound give that $p_j$ is at most $\binom{%
m}{j}$ times the probability that the three conditions hold and some
fixed set $S$ of $j$ vertices causes the graph to fail condition
$C_0$. We will bound this in three cases depending on the size of
$j$.

Recall that we defined $k=\frac{\ln \ln n}{2p}$.  Note that by our lower bound on $p$ we have $k=o(n)$.

Case 1: ${\frac{1 }{p \sqrt{\ln n}}}\leq j \leq k.$

We will show that there are almost surely no non-nice subsets at all
in this range, minimal or otherwise.  We can give an upper bound on the probability that a particular set $S$ of $j$ vertices fails to be nice by the probability that no vertex outside $S$ has either 0 or 1 neighbors in $S$ (we restrict our search to vertices outside $S$ so we do not have to worry about the presence or absence of self-loops).

Direct computation of the
probability that a fixed set of $j$ vertices has either 0 or 1
vertices adjacent to exactly one vertex in the set gives:
\begin{eqnarray*}
p_j &\leq& \binom{m}{j} (
(1-j p(1-p)^{j-1})^{m-j}+m j p(1-p)^{j-1}(1-j p(1-p)^{j-1})^{m-j-1})\\
&\leq& (m e p \sqrt{\ln n})^j ((1-j p(1-p)^{j-1})^{m-j}+m j
p(1-p)^{j-1}(1-j
p(1-p)^{j-1})^{m-j-1}) \\
&\leq& (m e p \sqrt{\ln n})^j ((1-j p e^{-j p(1+o(1))})^{m-j}+m j p
e^{-jp(1+o(1))} (1-j p e^{-j p
(1+o(1))})^{m-j-1}) \\
&\leq& ((1+o(1)) m e p \sqrt{\ln n})^j (e^{-m j p(1+o(1)) e^{-j p
(1+o(1))}})(1+(1+o(1))m j p e^{-j p (1+o(1))}).
\end{eqnarray*}
It follows from our bounds on $j$ and $p$ that $m j p e^{-jp}$ tends
to infinity, so the second half dominates the last term of the above
sum and we have:
\begin{equation*}
p_j \leq (1+o(1))(m e p \sqrt{\ln n})^j (e^{-m j p(1+o(1)) e^{-j p
(1+o(1))}})(2 m j p e^{-j p (1+o(1))}).
\end{equation*}
Taking logs and using $\delta n \leq m \leq n$ gives:
\begin{eqnarray*}
\ln(p_j) &\leq& (1+o(1))j (\ln (e n p \sqrt{\ln n}) - \delta n p
e^{-j p
(1+o(1))} - p + {\frac{\ln (2 n j p) }{j}}) \\
&\leq& (1+o(1))j (4\ln(np) - \delta n p e^{-k p (1+o(1))}) \\
&=& (1+o(1)) j (4 \ln (np) - {\frac{\delta n p }{(\ln n)^{{\frac{1 }{2}%
}+o(1)}}}).
\end{eqnarray*}
Since $n p>\frac{\ln n}{s}$, taking $n$ large gives that the
probability of failure for any particular $j$ in this range is
$o(1/n^4)$, and adding up over all $j$ and $m$ gives that the
probability of a failure in this range is $o(1/n^2)$.

Case 2: $1 \leq j \leq {\frac{1 }{p \sqrt{\ln n}}}$.

Let $b$ be the number of vertices outside $S$ adjacent to at least one vertex in
$S$, and let $a$ be the number of edges between $S$ and the vertices
of $G$ outside $S$.  For $j \geq s$ let $E_j$ be the event that the graph
satisfies conditions $C_1$ through $C_3$ but some fixed set $S$ of
$j$ vertices fails to be nice.

By Bayes' theorem we have
\begin{equation*}
p_j \leq \binom{m}{j} \sum_{w=0}^{nj} \P(E_j | a=w) \P(a=w)
\end{equation*}
We bound the terms in two subcases depending on the size of $w$
relative to $j$.

Case2a: $w<10j$. The claim here is that in this range it is
impossible for $E_j$ to occur.

Let $G^2$ denote a graph with on the same vertex set as $G$, but with $i$ connected to $j$ in $G^2$ iff $i$ and $j$ are of distance at most 2 in $G$.

As in Lemma \ref{lemma:expansion}, if $G$ is  locally sparse then
$S$ must have at most $\frac{18j}{\ln \ln n}$ vertices of degree at
most $\ln \ln n$.  Condition W1 now implies that at most
$\frac{18sj}{\ln \ln n}$ vertices of $S$ can lie in the same
component of the induced subgraph of $G^2$ on $S$ as a vertex of
degree at least $\ln \ln n$.  Thus the induced subgraph of $G^2$ must contain a component $S_1$ which does not contain a vertex having degree in $G$ at least $\ln \ln n$.  $S_1$ must have size at most $s-1$ by condition W1.  Note that this component shares no neighbors in $G$ with the rest of $S$.

Suppose that every vertex in $G$ adjacent to $S_1$ had two neighbors
in $S_1$.  It would then follow that the induced subgraph of $G$ on
$S_1 \cup N(S_1)$ contained at least $$2|N(S_1)|-|S_1 \cap N(S_1)|$$
edges.  On the other hand, condition W2 implies that this induced
subgraph is a forest, so has at most
$$|S_1 \cup N(S_1)|-1=|S_1|+|N(S_1)|-|S_1 \cap N(S_1)|-1$$
edges.  Combining these two inequalities would yield $|N(S_1)| \leq
|S|-1$.

Thus either $S_1$ is a non-expanding subset of $S$ of size at most
$s-1$, or there is a vertex adjacent to exactly one vertex in $S_1$
(and thus in $S$).  This fulfills the second requirement for $G$ to
be good.

If $|S| \geq s$, then we can find a second component $S_2$
satisfying the same conditions as $S_1$.  Unless either $S_1$ or
$S_2$ fails to expand, there will be a vertex in $G$ adjacent to
exactly one vertex in $S_1$ and another vertex adjacent to exactly
one vertex in $S_2$, so $S$ must be nice.  Thus the first
requirement for $G$ to be good is also satisfied.

Case 2b: $w \geq 10j$.  If $S$ is not nice, then at least $b-1$ of
the neighbors of $S$ must be adjacent to at least two vertices in
$S$. This implies that $b \leq \frac{a+1}{2}$.  It follows that we
can bound $\P(E_j)$ by ${\hbox{\bf P}}(b \leq \frac{w+1}{2}|a=w)$.
To do this, we fix a set of  $\frac{w+1}{2}$ vertices and bound  the
probability that $w$ consequentially randomly selected vertices were
in that set. Using the union bound over all possible sets of
$\frac{w+1}{2}$ vertices, we obtain
\begin{eqnarray*}
{\hbox{\bf P}}(b \leq \frac{w+1}{2}|a=w) &\leq&
\binom{m-j}{\frac{w+1}{2}} (\frac{w+1}{2(m-j)})^w \\
&\leq& (\frac{2 e (m-j)}{w-1})^{\frac{w+1}{2}} (\frac{w+1}{2(m-j)})^w \\
&\leq& (\frac{4w}{m})^{\frac{w-1}{2}}.
\end{eqnarray*}
This bound is decreasing in $w$ for the entire range under
consideration (our bounds on $j$ guarantee $w$ is at most $\frac {10
n}{\sqrt{\ln n}}$).  Therefore we can bound $\P(q_j | a=w)$ by the
probability given $a=10j$, giving
\begin{eqnarray*}
p_j &\leq& 3 \sqrt{n} \binom{m}{j}(\frac{40 j}{m})^{5j}\\
&\leq& 3 \sqrt {n} (\frac{m e}{j})^j (\frac{40 j}{m})^{5j}\\
&\leq& 3 \sqrt{n} (\frac{130j}{n^{\prime}})^{4j}.
\end{eqnarray*}
This bound is decreasing in $j$ in the range under consideration,
and plugging in $j=s$ gives that $p_j=o(1/n^4)$ for each $j$ and $m$
in the range.  By the union bound the probability of failure in this
range is $o(1/n^2)$.

\section{Proofs of Lemmas \protect\ref{lemma:singaug} and \protect\ref%
{lemma:nonsingaug}}

\subsection{Proof of Lemma \ref{lemma:singaug}}
Let $Q$ be a fixed nice matrix which is not $s-1$-saturated.  Since
$Q$ is not saturated, there is a vector $v:=(v_1, v_2, \dots v_m)^T$
in the nullspace of $Q$ whose support does not contain the
coordinates corresponding to any non-expanding subset of at most
$s-1$ rows.  Let $D$ be the number of nonzero coordinates in $v$. If
$D$ were at most $k$, then the goodness of $Q$ would guarantee that
the support of $v$ would be nearly nice, meaning that some vertex
$i$ has only one neighbor in the support of $v$.  But this is a
contradiction, as the product of row $i$ with $v$ would then be
$v_i$, which is by assumption not 0 since it is in the support of
$v$.

We may therefore assume that $D>k$ and, without loss of generality,
that it is the first $D$ coordinates of $v$ which are nonzero.  We
now consider the linear equation

\begin{equation} \label{eqn:singaug}
\sum_{i=1}^D v_i x_i =0
\end{equation}

If the new column $x$ does not satisfy this equation, then
augmenting $Q$ by this column will increase the rank by 1 and (by
the symmetry of $Q$) augmenting $Q$ simultaneously by the $x$ and
its transpose will increase the rank by 2.  Therefore it suffices to
bound the probability that \eqref{eqn:singaug} is satisfied.  To do
so, we use the following variation of the classical
Littlewood-Offord Theorem \cite{Erdos} due to Hal\'{a}sz \cite{Hal}.

\begin{theorem} \label{thm:linearLO}
Let $\chi_i$ be independent random variables satisfying
\begin{equation*}
 \max_i \sup_{c \in \R} \P(\chi_i=c) \leq 1-\rho.
\end{equation*}
Then
\begin{equation*}
\P(\sum_{i=1}^D \chi_i =0)=O((D\rho)^{-1/2})
\end{equation*}
\end{theorem}

Although the $x_i$ aren't all random in our case (our conditioning
on the normality of $(G_m, G_{m+1})$ guarantees that all variables
corresponding to vertices in non-nice subsets of $G$ are 0), most of
them will be random.  Since we are assuming $G$ to be good, the
number of non-random $x_i$ is bounded above by the number of vertices
of degree at most $s$, which is in turn bounded above by $\frac{1}{p
\ln n}=o(D)$.

Thus after removing the $x_i$ which are forced to be 0 we are left
with $D(1+o(1))$ independent variables with nonzero coefficients.  Each $v_i x_i$ is equal to 0 with probability $1-p$ and $v_i w_i$ with probability $p$.  Applying Theorem \ref{thm:linearLO} with $\rho=p$, we see that the probability \eqref{eqn:singaug} is
satisfied is therefore at most $O((Dp)^{-1/2})=O((kp)^{-1/2})$.

\subsection{Proof of Lemma \ref{lemma:nonsingaug}}
Just as in the proof of Lemma \ref{lemma:singaug}, the goal will be eventually to use a Littlewood-Offord Lemma to show that a certain expression is almost surely not equal to 0.  In this case, the expression in question will be the determinant of an appropriately chosen submatrix of the augmentation of $Q$.  Before we apply it, however, we first attempt to obtain additional information on the structure of $Q$, and in particular on $G=G(Q)$.

Let $T$ be the union of all minimal non-expanding subsets of $G$
which contain at most $s-1$ vertices.  We begin by proving two
lemmas on the structure of $T$ under the assumption that $G$ is
well-separated.

\begin{lemma} \label{lemma:Tcycles}
Let $T$ be defined as above.
Then $N(T) \cap T = \emptyset$.
\end{lemma}
\begin{proof}
Assume to the contrary that there are two vertices $v$ and $w$ in
$T$ which are adjacent in $T$.  Let $T_v$ and $T_w$ be minimal
non-expanding subsets of $G$ such that $T_v$ contains $v$, $T_w$
contains $w$, and $|T_v|$ and $|T_w|$ are both at most $s-1$ ($T_v$
and $T_w$ may in fact be equal here).  By the minimality of $T_v$,
we have $|N(T_v-v)| \geq |T_v|-1 \geq |N(T_v)|$, so it follows that
$w$ must also be in $N(T_v-v)$.  A similar argument shows that any
vertex in $T_v$ with a neighbor in $T_w$ has at least one other
neighbor in $T_w$, and vice versa.

We now consider the graph on vertex set $T_v \cup T_w$ whose edges
correspond to edges in $G$ with one vertex lying in $T_v$ and the other vertex in $T_w$ (one or both of these vertices may be in $T_v \cap T_w$).  By the above
argument, no vertex in this graph has degree exactly 1.  However,
the well-separation condition W2 guarantees that this graph is a
forest (since any vertex in $T$ has degree at most $s-1$ in $G$),
which implies that the graph must in fact be empty.
\end{proof}

\begin{lemma}\label{lemma:matching}
Let $T$ be defined as above.  Then there is a $T_1 \subset T$ with $|T_1|=|N(T)|$ such
that $G$ has exactly one matching between $T_1$ and $N(T)$.
\end{lemma}
\begin{proof}
Consider the graph on $T \cup N(T)$ which includes all edges with at
least one endpoint in $T$.  We perform the following algorithm to
construct $T_1$: At each step we pick an unused vertex in $T$ with
exactly one unused neighbor in $N(T)$.  We then add that vertex to
$T_1$ and consider both it and its neighbor as used.

Assuming this algorithm eventually matches every vertex in $N(T)$
with a vertex in $T$, we are done, since the uniqueness of the
matching is clear from our matching process.  Showing the algorithm
does not terminate prematurely is equivalent to showing that after
every step either every vertex in $N(T)$ is used or there is an
unused vertex in $T$ with exactly one unused neighbor in $N(T)$.

To do this, we first note that any unused vertex in $N(T)$ has at
least two unused neighbors in $T$ (the argument in the previous
lemma shows that it has at least two neighbors in $T$, and by
construction our algorithm marks a vertex in $N(T)$ as used as soon
as its first neighbor in $T$ is used).  Furthermore, by well
separation the induced subgraph on the unused vertices of $T$ and
$N(T)$ is a forest, which is nonempty unless all vertices of $N(T)$
have been used.  It therefore must have a vertex of degree one,
which must be in $T$ since every unused vertex in $N(T)$ has degree
at least two.  This allows us to continue the algorithm until all of
$N(T)$ is matched.
\end{proof}

Without loss of generality we can now view $A(G)$ as the block matrix
below

\begin{equation*}
A(G)=\left(%
\begin{array}{cccc}
  A(G\backslash (T \cup N(T))) & A(G\backslash (T \cup N(T)), N(T)) & 0 & 0 \\
  A(N(T),G\backslash (T \cup T_1)) & A(N(T)) & A(N(T),T_1) & A(N(T),T \backslash T_1) \\
  0 & A(T_1,N(T)) & 0 & 0 \\
  0 & A(T \backslash T_1, N(T)) & 0 & 0 \\
\end{array}%
\right),
\end{equation*}

where $A(G,H)$ denotes the adjacency matrix of the induced bipartite
subgraph between $G$ and $H$.  The blocks in the lower right are 0
because of Lemma \ref{lemma:Tcycles}.  By construction the fourth
row of blocks is contained in the span of the third row.

Let $B$ be the matrix formed by the first three rows and columns of
blocks of $A$ (note that $\rank(B)=\rank(A))$.  To prove
\ref{lemma:nonsingaug} it suffices to show that augmentation will
almost surely increase the rank of $B$.  Our assumption of normality
guarantees we can think of the augmentation as

\begin{equation*}
B'=\left( \begin{array}{cccc}
  A(G\backslash (T \cup N(T))) & A(G\backslash (T \cup N(T)), N(T)) & 0 & x \\
  A(N(T),G\backslash (T \cup T_1)) & A(N(T)) & A(N(T),T_1) & y \\
  0 & A(T_1,N(T)) & 0 & 0 \\
  x^T & y^T & 0 & 0 \\
\end{array}%
\right),
\end{equation*}

where $x$ and $y$ are random vectors each of whose entries are 1
with probability $p$, 0 otherwise.  We now expand $\det(B')$ by
minors simultaneously along all rows and columns in the third row
and column of blocks.  By Lemma \ref{lemma:matching}, only one
nonzero term remains, so we are left with

\begin{equation*}
\det(B')=\pm \det \left( \begin{array}{cc}
  A(G\backslash (T \cup N(T))) &  x \\
  x^T & 0 \\
\end{array}%
\right) =\pm \sum_{i=1}^m \sum_{j=1}^m A(i,j) x_i x_j,
\end{equation*}
where $A(i,j)$ denotes the $(i,j)$ cofactor of $A(G\backslash (T
\cup N(T)))$.

This is the expression that we are aiming to show is almost surely non-zero.  Unlike in the proof of Lemma \ref{lemma:singaug}, however, this expression is a quadratic form in the variables $x_i$.  Thus instead of using the original Littlewood-Offord Lemma, we will use the following quadratic version proved in \cite{CV}.

\begin{lemma}
\label{lemma:modifiedquadraticL-O} Let $a_{ij}$ be fixed constants
such that there are at least $q$ indices $j$ such that for each $j$
there are at least $q$ indices $i$ for which $a_{ij} \neq 0$. Let
$z_1, z_2, \dots z_n$ be as in Lemma \ref{thm:linearLO}.
Then for any fixed $c$
\begin{equation}  \label{eqn:qref}
{\hbox{\bf P}}(\sum_{i=1}^n \sum_{j=1}^n a_{ij} z_i
z_j=c)=O((q \rho)^{-1/4}),
\end{equation}

where the implied constant is absolute.
\end{lemma}

Our goal will be to show that probably enough of
these cofactors are nonzero that we can apply this Quadratic
Littlewood-Offord Lemma to say that the determinant of $B'$ is
probably not zero.  To do so, we first establish some properties of
the matrix $C:=A(G\backslash (T \cup N(T)))$.

\begin{lemma} \label{lemma:nonsingC}
$C$ is nonsingular.
\end{lemma}
\begin{proof}
Any subset of $T_1$ must be expanding due to the matching between
$T_1$ and $N(T)$.  This implies that the first three rows of blocks
of $A$ cannot contain any non-expanding subsets of size at most
$s-1$ (such a subset would have to be in $T$ by the definition of
$T$, but could not be entirely within $T_1$ since $T_1$ expands.
Since $A$ is $(s-1)-$saturated, it follows that the rank of $A$ is
at least $n-|(T\backslash T_1)|$.

On the other hand, we know the third row of blocks is independent (the uniqueness of the matching in Lemma \ref{lemma:matching} implies $A(T_1, N(T))$ has determinant $\pm 1$), and contains the fourth row of blocks in its span.  Since this already accounts for the entire nullspace of $A$, the first three rows of blocks of $A$ must be independent.

This implies we can perform row reduction to eliminate the $A(G\backslash (T \cup N(T)), N(T))$ block of $A$, and that the rows of the reduced matrix (including the rows of $C$) are still independent.
\end{proof}

\begin{lemma} \label{lemma:nearniceness}
$G \backslash (T \cup N(T))$ is "almost good" in the following sense:

1. Every minimal non-nice subset of the vertices of $G \backslash (T \cup N(T))$ has size either at most $s-1$ or at least $k-\frac{1}{p \ln n}$.

2. Every minimal non-nearly nice subset of the vertices of $G \backslash (T \cup N(T))$ has size at least $k-\frac{1}{p \ln n}$.

3. At most $\frac{1}{p \ln n}$ vertices of $G \backslash (T \cup N(T))$ have degree less than
$s$.

\end{lemma}
\begin{proof}
Let $S$ be a subset of the vertices of $G\backslash (T \cup N(T))$
of size at most $k-\frac{1}{p \ln n}$, and let $S_1$ denote those
vertices in $N(T)$ which have exactly one neighbor in $S$.  We now
perform the following algorithm: So long as $S_1$ remains nonempty,
we choose a vertex $v$ in $S_1$, choose a vertex of $T$ adjacent to
$v$ which is not already in $s$, add that vertex in $T$ to $S$, and update $S_1$ accordingly.  Since each
vertex in $N(T)$ has at least two neighbors in $T$, we will always
be able to continue this process so long as $S_1$ remains nonempty.
In particular, the process must terminate by the time we have added
all the vertices in $T$ to $S$.

Let $S'$ be the set which results once the algorithm terminates. We
first note that $S'$ can have size at most $k$ (The vertices in $T$
which we add to $S$ always have degree at most $s-1$, and the number
of such vertices is bounded by our assumption that $G$ is good).
Furthermore, there is a natural matching between $S' \cap T$ and
$N(T)$ given by matching each vertex of $S' \cap T$ with the $v$ in
$N(T)$ which caused it to be added to $S'$.  This implies that $S'
\cap T$ (and thus $S'$) does not contain any non-expanding subset of
size at most $s-1$.  Since $G$ is good, this implies that $S'$ must
be nearly-nice, meaning there is some $w$ with only one neighbor in
$S'$.  This $w$ can't be in $N(T)$ by construction, and it can't be
in $T$ since $S'$ contains no vertices from $N(T)$.  It follows
that $w$'s neighbor must be in $S$, so $S$ is also nearly-nice.

A similar argument gives that all minimal non-nice subsets in
$G\backslash (T \cup N(T))$ either have size at least
$k-\frac{p}{\ln n}$ or at most $s-1$.  To show there aren't many
vertices of degree at most $s$ in $G\backslash (T \cup N(T))$, we
first note that by condition $W_1$ of well separation a given vertex
can only have $s-1$ neighbors in $N(T)$.  It follows that any vertex
of degree at most $s$ in $G\backslash (T \cup N(T))$ had degree at
most $2s$ in $G$, and this can be bounded by the same argument as
\eqref{eqn:lowdegree}.
\end{proof}

Since $C$ has full rank, dropping any of the columns of $C$ will lead to a $%
m \times m-1$ matrix whose rows admit (up to scaling) precisely one
nontrivial linear combination equal to 0. If any of the rows in that combination are dropped, we will be left with an $%
m-1 \times m-1$ nonsingular matrix, i.e. a nonzero cofactor.

As in Lemma \ref{lemma:singaug}, the rows with nonzero coefficients in this linear combination must form a non-nearly nice subset of the rows of the column deleted matrix of $C$.  By Lemma \ref{lemma:nearniceness} the only way that the combination can involve fewer than $k-\frac{1}{p \ln n}$ rows is if the rows involved formed a non-nice subset of $G \backslash (T \cup N(T))$, one of whose neighbors was the removed column.  We can upper bound the number of columns whose removal could possibly cause this difficulty by the number of vertices in $G \backslash (T \cup N(T))$ which have at least one neighbor which has degree at most $s$ in $G \backslash (T \cup N(T))$, which by Lemma \ref{lemma:nearniceness} is  $O(\frac{s}{p \ln n})=o(n)$.

Dropping any other column will lead to many nonzero cofactors, so we can apply the Quadratic Littlewood Offord Lemma with $q=k-\frac{1}{p \ln n}=k(1-o(1))$ and $\rho=p$ to bound the probability that the determinant is 0, proving Lemma \ref{lemma:nonsingaug}.

 \section{Proofs of Theorem \ref{thm:betterexpression} and \ref{thm:altexpression}}
\textbf{Proof of Theorem \ref{thm:betterexpression}:} What we will show here is that every $(s-1)-$saturated good matrix satisfies the conclusion of the theorem.  This is sufficient since by Theorem \ref{thm:mainresult} and \ref{lemma:goodmatrices} the graph will be both $(s-1)-$saturated and good with probability $1-O(\ln \ln n^{-1/4})$.  We will prove this result by contradiction.

Suppose that some subset $R$ of the rows of $Q_n$ is minimally
dependent but does not contain a non-expanding subset of size at
most $s-1$.  Since $G$ is good, it much be true that $|R|>\frac{\ln
\ln n}{p}$.

Since $G$ is $(s-1)-$saturated, there must be a subset $S_0 \in R$
which is both independent and maximal subject to not containing any
non-expanding sets of size at most $s-1$.  In particular, any row in
$R$ which is not in $S_0$ would create a non-expanding set when
added to $S_0$.  Since non-expanding sets are dependent, any row in
$R \backslash S_0$ can be written as a linear combination of at most
$s-1$ rows of $S_0$.

Now by assumption the rows of $R$ satisfy some linear relationship
\begin{equation}\label{eqn:Wdependence}
\sum_{i \in W} a_i v_i=0.
\end{equation}

For each row which is in $R$ but not in $S_0$, we substitute the corresponding linear combination of at most $s-1$ rows in $S_0$ which equals it into \eqref{eqn:Wdependence}.  This yields a linear relationship between the rows of the independent set $S_0$, which must therefore have all its coefficients equal to zero.  There were initially at least $\frac{\ln \ln n}{p}$ nonzero coefficients in \eqref{eqn:Wdependence}, and each substitution can change at most $s$ of them to zero.  It follows that

\begin{equation*}
|R \backslash S_0| \geq \frac{\ln \ln n}{s p}.
\end{equation*}

Each vertex in $R \backslash S_0$ is part of a non-expanding set of
size at most $s-1$, and thus has degree at most $s-2$.  However, $G$
is by assumption good, so has at most $\frac{1}{p \ln n}$ vertices
of degree this small.  This is a contradiction, so Theorem
\ref{thm:betterexpression} is proved.

\textbf{Proof of Theorem \ref{thm:altexpression}:} Again, we will show that any $(s-1)$-saturated, good matrix satisfies the conclusion of the theorem.

On one hand, it is clear that, for any $S$, the expression $n-|S|+|N(S)|$ is an upper bound for the rank of $A(G)$.  Thus it suffices to exhibit some set $S$ for which that expression is at most as large to the rank.  To do so, we return to the block decomposition of the proof of Lemma \ref{lemma:nonsingaug}.  Note that by the proof of Lemma \ref{lemma:nonsingC} the rows of the first three blocks of $A$ in this decomposition are independent, so we have
\begin{equation*}
\rank(A) \geq n-|T \backslash T_1|.
\end{equation*}

Conversely, if we take $S=T$ then we have
\begin{equation*}
n-|S|+|N(S)|=n-|T|+|T_1|=n-|T \backslash T_1|,
\end{equation*}

where the first inequality comes from the matching between $N(S)$ and $T_1$.  Combining these two equations yields the desired result.

\section{Extensions, Open Problems and Avenues for Further Research}
Although the symmetric model $Q(W,p)$ is natural from a graph theoretic viewpoint, it is also of interest to consider what happens when the matrix $W$ is no longer symmetric.  In the case $s=1$, we can show via a similar argument to the one in this paper that the following result holds:
\begin{theorem}
 Let $W$ be any (symmetric or non-symmetric) matrix of weights which is non-zero off the main diagonal.  Let $\frac{(1+\epsilon) \ln n}{n} < p < \frac{1}{2}$, and let $\xi_{ij}$ be independent Bernoulli variables which are 1 with probability $p$ and 0 otherwise.  Let $A(W,p)$ be the matrix whose entries are given by $a_{ij}=w_{ij}\xi_{ij}$.  Then
\begin{equation*}
 \P(A(W, p) \textrm{ is singular }=O((\ln \ln n)^{-1/2})
\end{equation*}
\end{theorem}
It seems more difficult to obtain a workable extension of the full statement of Theorem \ref{thm:betterexpression} to non-symmetric models.  For example, suppose we were to look at a matrix $A$ whose entries were 1 with probability $\frac{0.8 \ln n}{n}$ and 0 otherwise.  With high probability this matrix will have rows and columns that are entirely 0, and furthermore the number of nonzero rows and columns of $A$ will likely not be equal.  If $A$ has more nonzero rows than columns, then that would imply some sort of dependency among the nonzero rows of $A$.  However, there does not seem to be any obvious way to describe this dependency.

The assumption in Corollary \ref{cor:randomw} that the entries of $A$ have equal probabilities of being 0 seems quite artifical, and it would be of interest to replace that assumption by the assumption that no entry is too likely to take on the value 0.  More generally, we have the following conjecture
\begin{conjecture} Let $A$ be a random symmstric matrix whose above-diagonal entries $a_{ij}$ are independent, random variables satisfying
 \begin{equation*}
  \sup_{i, j} \sup_{c \in \C} \P(a_{ij}=c) \leq 1-\frac{(1+\epsilon) \ln n}{n}.
 \end{equation*}
Then $A$ is almost surely non-singular.
\end{conjecture}
The methods of \cite{CTV} give that this conjecture is true if $\frac{(1+\epsilon) \ln n}{n}$ is replaced by $n^{-1/2+\epsilon}$.  Conversely, the methods in this paper and \cite{CV} cover the case where the entries are allowed to concentrate on the particular value of 0 with high (equal) probability.  However, there seems to be no natural reason why the singularity probability should be less when the variables concentrate on 0 as opposed to some other value.

Alternatively, we can consider the situation where the matrix $W$ which is being sparsified already has certain entries set equal to 0.  Here the intutition is that if $W$ does not have too many entries already set equal to 0, then the results of this paper should also hold for the sparsified matrix.  This leads to the following conjecture
\begin{conjecture}
 Let $W$ be a symmetric matrix containing at least $m$ nonzero entries in each row.  Then if $\frac{(1+\epsilon) \ln n}{m}<p<\frac{1}{2}$, then $Q(W,p)$ is almost surely nonsingular.
\end{conjecture}
This conjecture and a short coupling argument would be sufficient to handle the case of matrices whose entries are 0 with nonequal probability.

To motivate our final conjecture, we again turn to the linkage between expansion of a graph $G$ and dependencies in its adjacency matrix $A(G)$.  One way of thinking of our theorems is that, even for our very sparse graphs, non-expansion is the only source of dependency in the adjacency matrix.  If every subset of a set $S$ of rows expands well, then $S$ will be independent.  Although random regular graphs are very sparse, they also typically will expand very well.  Is it then the case that the adjacency matrix of a random $d-$regular graph will almost surely be nonsingular?

This is not the case for $d=2$ (as the graph will almost surely contain a $4k$-cycle for some $k$, which implies a singular adjacency matrix), but it seems likely to be true for larger $d$.
\begin{conjecture}
For any fixed $d>2$, the adjacency matrix of a random $d-$regular graph is almost surely non-singular.
\end{conjecture}

\end{document}